\documentclass[11pt]{article}

\usepackage[utf8]{inputenc} 
\usepackage{amsmath}
\usepackage{graphicx}
\usepackage{subcaption}
\usepackage[top=30mm, bottom=30mm, left=27mm, right=27mm]{geometry}
\usepackage{amsfonts}
\usepackage{mathtools}
\usepackage{amssymb}
\usepackage{amsthm}
\usepackage{chngpage}
\usepackage{hyperref}
\hypersetup{
	colorlinks=true,
	linkcolor=black,
	citecolor=black,
	filecolor=black,      
	urlcolor=black,
}
\usepackage{enumerate}
\usepackage{makeidx}
\usepackage{bm}
\usepackage{thmtools}
\usepackage{thm-restate}
\usepackage{tikz}

\makeindex

\makeatletter
\renewenvironment{proof}[1][\proofname] {\par\pushQED{\qed}\normalfont\topsep6\p@\@plus6\p@\relax\trivlist\item[\hskip\labelsep\bfseries#1\@addpunct{.}]\ignorespaces}{\popQED\endtrivlist\@endpefalse}
\makeatother

\newcommand{\PP}{\mathbb{P}}
\newcommand{\UU}{U^{\textnormal{all}}}
\newcommand{\rank}{\operatorname{rank}}

\newtheorem{proposition}{Proposition}[section]
\newtheorem{conjecture}[proposition]{Conjecture}
\newtheorem{lemma}[proposition]{Lemma}
\newtheorem{theorem}[proposition]{Theorem}

\newtheorem{question}[proposition]{Question}

\theoremstyle{definition}

\newtheorem*{remark*}{Remark}
\newtheorem*{theorem*}{Theorem}

\mathcode`l="8000
\begingroup
\makeatletter
\lccode`\~=`\l
\DeclareMathSymbol{\lsb@l}{\mathalpha}{letters}{`l}
\lowercase{\gdef~{\ifnum\the\mathgroup=\m@ne \ell \else \lsb@l \fi}}%
\endgroup

\title{Partial shuffles by lazy swaps}
\author{Barnabás Janzer\thanks{Department of Pure Mathematics and Mathematical Statistics, University of Cambridge, Wilberforce Road, Cambridge CB3 0WB, United Kingdom. Email: bkj21@cam.ac.uk. Supported by EPSRC DTG.} \and J. Robert Johnson\thanks{School of Mathematical Sciences, Queen Mary, University of London, London E1 4NS, United Kingdom. E-mail: r.johnson@qmul.ac.uk.} \and Imre Leader\thanks{Department of Pure Mathematics and Mathematical Statistics, University of Cambridge, Wilberforce Road, Cambridge CB3 0WB, United Kingdom. Email: i.leader@dpmms.cam.ac.uk.}
}
\date{\vspace{-21pt}}

\begin{document}
	\maketitle
\begin{abstract}
What is the smallest number of random transpositions (meaning that we swap given pairs
of elements with given probabilities) that we can make on an $n$-point set to ensure that each element is uniformly distributed -- in the sense that the probability that $i$ is mapped to $j$ is $1/n$ for all $i$ and $j$? And what if we insist that each {\it pair} is uniformly distributed?

In this paper we show that the minimum for the first problem is about
$\frac{1}{2} n \log_2 n$, with this being exact when $n$ is a power of 2. For the second problem, we show that, rather surprisingly, the answer is not quadratic: $O(n \log^2 n)$ random transpositions suffice. We also show that if we ask only that the pair $1,2$ is uniformly distributed then the answer is $2n-3$. This proves a conjecture of Groenland, Johnston, Radcliffe and Scott.
\end{abstract}
\section{Introduction}

Let $S_n$ denote the symmetric group on $n$ elements, i.e., the set of bijections $[n]\to[n]$. A \emph{lazy transposition} with parameters $(a,b,p)$ is a random permutation $T$ such that $T=(a,b)$ with probability $p$ and $T$ is the identity with probability $1-p$, where $(a,b)$ denotes the transposition swapping $a$ and $b$. We say that the independent lazy transpositions $T_1,\dots, T_l$ form a \emph{transposition shuffle} (of order $n$ and length $l$) if their product $T_1\dots T_l$, which is a random permutation of $S_n$, is uniformly distributed among all the $n!$ elements of $S_n$. What is the shortest possible length $U(n)$ of a transposition shuffle of order $n$? This problem was first raised by Fitzsimons~\cite{fitzsimons}, and also independently studied by Angel and Holroyd~\cite{angel2018perfect}.

It is not difficult to show that $U(n)\leq \binom{n}{2}$ -- there are many constructions achieving this. One example is obtained as follows. We may inductively take $\binom{n-1}{2}$ lazy transpositions $T_1,\dots,T_{\binom{n-1}{2}}$, only permuting $\{1,\dots,n-1\}$, which form a transposition shuffle of order $n-1$. We may also construct lazy transpositions $T_1', \dots, T_{n-1}'$ such that $T_1'\dots T_{n-1}'$ maps the element $n$ to each of $1,2,\dots,n$ with probability $1/n$. Then the random permutation $T_1'\dots T_{n-1}'T_1\dots T_{\binom{n-1}{2}}$ is easily seen to be uniform. Fitzsimons~\cite{fitzsimons} and Angel and Holroyd~\cite{angel2018perfect} asked whether the upper bound $\binom{n}{2}$ is tight. Very recently, Groenland, Johnston, Radcliffe and Scott~\cite{groenland2022perfect} proved the following theorem, answering this question in the negative.

\begin{theorem}[Groenland, Johnston, Radcliffe and Scott~\cite{groenland2022perfect}]\label{theorem_GJRS}
	For all $n\geq 6$ we have $U(n)<\binom{n}{2}$. In fact, $$U(n)\leq \frac{2}{3}\binom{n}{2}+O(n\log n).$$
\end{theorem}

A simple lower bound for $U(n)$ can be obtained from the observation that the product of $l$ lazy transpositions takes at most $2^l$ possible values in $S_n$, giving $U(n)\geq \log_2(n!)=\Theta(n\log n)$. Surprisingly, this is the best known lower bound, even though the argument above ignores uniformity and only uses that each permutation can be reached with positive probability. In fact, if we only want to achieve that the final permutation is `close' to uniform, but not necessarily exactly uniform, then it is enough to take $O(n\log n)$ lazy transpositions, as shown by Czumaj~\cite{czumaj2015random}.

While the result of Groenland, Johnston, Radcliffe and Scott~\cite{groenland2022perfect} answered the question of Fitzsimons~\cite{fitzsimons} and Angel and Holroyd~\cite{angel2018perfect}, and showed that the natural upper bound $\binom{n}{2}$ is not tight, there is still a large gap between the best known upper and lower bounds for $U(n)$. The authors of~\cite{groenland2022perfect} conjectured that in fact $U(n)=o(n^2)$, and also asked about improving the lower bound $\Omega(n\log n)$.

A natural approach to better understand the numbers $U(n)$ is to consider, for some integer $k\leq n$, the smallest possible number of (independent) lazy transpositions $T_1,\dots,T_l$ such that the product $T_1\dots T_l$ maps the elements $1,\dots,k$ uniformly to the $n(n-1)\dots (n-k+1)$ possible ordered $k$-tuples. Such a sequence is called a \emph{$(k,n)$-shuffle}, and the shortest possible length of a $(k,n)$-shuffle is denoted $U_k(n)$.


Note that we have $U_1(n)=n-1$ for all $n$. Applying this fact repeatedly (similarly to how we derived the bound $U(n)\leq \binom{n}{2})$, we easily get

\begin{equation}\label{equation_knshufflebound}
U_k(n)\leq kn-\binom{k+1}{2}.
\end{equation}
Moreover, any improvement on this upper bound for $U_k(n)$ gives an upper bound better than $\binom{n}{2}$ for $U(n)$. Groenland, Johnston, Radcliffe and Scott~\cite{groenland2022perfect} managed to improve~\eqref{equation_knshufflebound} for all $k\geq 3$ (and used this to obtain their upper bound for $U(n)$ in Theorem~\ref{theorem_GJRS}). 
For the case $k=2$, they conjectured that \eqref{equation_knshufflebound} cannot be improved.

\begin{conjecture}[Groenland, Johnston, Radcliffe and Scott~\cite{groenland2022perfect,groenland2022short}]\label{conjecture_knshufflefor2}
For all $n\geq 2$ we have $$U_2(n)=2n-3.$$
\end{conjecture}

Towards Conjecture~\ref{conjecture_knshufflefor2}, Groenland, Johnston, Radcliffe and Scott~\cite{groenland2022short} proved that if we only use lazy transpositions of the form $(1,b,p)$, then we need at least $1.6n-O(1)$. Moreover, they proved that if instead of uniformity we only ask for \emph{reachability}, i.e., if we only want to achieve that the elements $1,2$ 
may end up at any other pair $(i,j)$, then the minimal number of lazy transpositions required is $\lceil 3n/2\rceil-2$  -- demonstrating a gap between the reachability problem and the uniformity problem for pairs.

Our first result in this paper is a proof of Conjecture~\ref{conjecture_knshufflefor2}.

\begin{theorem}\label{theorem_12uniform}
	For all $n\geq 2$ we have
	
	$$U_2(n)=2n-3.$$
\end{theorem}

Perhaps an even more natural question than determining $U_k(n)$ is as follows: rather than asking that one fixed $k$-tuple is mapped to each $k$-tuple with the same
probability, we ask that this should be true for \emph{every} $k$-tuple. More
precisely, given $1\leq k\leq n$, let $\UU_k(n)$ denote the shortest possible sequence of (independent) lazy transpositions $T_1,\dots,T_l$ such that whenever $1\leq x_1<x_2<\dots<x_k\leq n$, then the image of the $k$-tuple $(x_1,\dots,x_k)$ under $T_1\dots T_l$ is uniformly distributed among the $n(n-1)\dots(n-k+1)$ possible values. If $T_1,\dots,T_l$ are as above, let us say that they form a \emph{strong $(k,n)$-shuffle}. It is clear that $U_k(n)\leq \UU_k(n)\leq U(n)$ and $\UU_n(n)=U(n)$.

We start with the case $k=1$: thus we want to ensure that, for each $i$ and $j$, 
the probability that $i$ maps to $j$ is $1/n$. 
Our next result gives an essentially tight bound for $\UU_1(n)$. Note that
this bound, being $\Theta(n\log n)$, already matches the order of magnitude of the best known lower bound for $U(n)$.

\begin{theorem}\label{theorem_strongshufflefor1}
	For any positive integer $n$, we have
	$$\frac{1}{2}n\log_2 n\leq \UU_1(n)\leq \frac{1}{2}n\log_2 n+2n.$$
	Moreover, if $n$ is a power of $2$, then
	$$\UU_1(n)=\frac{1}{2}n\log_2 n.$$
\end{theorem}

We digress to mention an appealing geometric reformulation of this result. For the problem above, the natural `matrix of probabilities' has $i,j$ entry given by $\PP(i\textnormal{ maps to }j)$.
If we consider the columns of this matrix as vectors in $\mathbb{R}^n$, then it is easy to see that the problem is equivalent to the following. We start with $n$ independent vectors in $\mathbb{R}^n$, and at any step we replace two of the vectors, say $u$ and $v$, with the convex combinations $t u + (1-t) v$ and $(1-t) u + t v$ respectively. How many such steps are needed before we have mapped all the vectors to their centroid? The above theorem shows that the answer is about
$\frac{1}{2}n \log_2 n$. Interestingly, we do not see any `directly geometric' argument to establish this result, even approximately.

It is interesting to compare Theorem~\ref{theorem_strongshufflefor1} with the corresponding reachability problem, i.e., the minimal number of lazy transpositions required so that their product maps any $i$ to any $j$ with positive probability. This is equivalent to a well-known problem, sometimes called `gossiping dons' (see, e.g.,~\cite{bollobas2004extremal,bumby1981problem}), and the minimal number of lazy transpositions required is $2n-4$ for $n\geq 4$.

The next case is $k=2$: thus we are asking that every ordered pair maps to every other ordered pair with probability $\frac{1}{n(n-1)}$. Here 
one might expect that a quadratic number of lazy transpositions is needed. We prove that, surprisingly, $n(\log{n})^{O(1)}$ swaps still suffice.
\begin{theorem}\label{theorem_strongshufflefor2}
	We have
	$$\UU_2(n)=O(n\log^2 n).$$
\end{theorem}

Note that we clearly have $\UU_2(n)\geq \UU_1(n)$ for all $n$, so Theorem~\ref{theorem_strongshufflefor1} implies a lower bound $\UU_2(n)\geq \frac{1}{2}n\log_2{n}$. We have not managed to obtain any non-trivial improvement on this lower bound.

We can once again contrast Theorem~\ref{theorem_strongshufflefor2} with the corresponding reachability problem. As mentioned before, there exist lazy transpositions $T_1,\dots,T_l$ such that $l=\frac{3}{2}n+O(1)$ and $T_1\dots T_l$ maps $(1,2)$ to any pair $(i,j)$ (where $j\not =i$) with positive probability. Thus $T_l T_{l-1}\dots T_1$ maps each $(i,j)$ to $(1,2)$ with positive probability. So if we take an independent copy $T_1',\dots,T_l'$ of our sequence of lazy transpositions, we get that $T_1T_2\dots T_l T_l' T_{l-1}'\dots T_1'$ maps each $(i,j)$ to each $(i',j')$ with positive probability, giving an upper bound of $3n+O(1)$ for the reachability problem. It would be interesting to know what the asymptotic behaviour is.

\section{Tight bounds for \texorpdfstring{$(2,n)$}{(2,n)}-shuffles}
In this section we prove Theorem~\ref{theorem_12uniform} concerning lazy transpositions mapping the pair $(1,2)$ uniformly to all pairs. Our proof is motivated by a new proof of the lower bound for the corresponding reachability problem -- our short proof is different from the proof of Groenland, Johnston, Radcliffe and Scott~\cite{groenland2022short}, so we find it helpful to include it here. Let $R_2(n)$ denote the minimal number $l$ of transpositions $T_1,\dots,T_l$ such that for any $i$ and $j$ (distinct), there is a subsequence of $T_1,\dots,T_l$ whose product maps $(1,2)$ to $(i,j)$. (Equivalently, it is the minimal number of lazy transpositions such that their product maps $(1,2)$ to any $(i,j)$ with positive probability.)

\begin{theorem}[Groenland, Johnston, Radcliffe and Scott~\cite{groenland2022short}]\label{theorem_12reachable}
	For all $n\geq 2$ we have $$R_2(n)=\lceil 3n/2\rceil -2.$$
\end{theorem}

For reasons that will soon become clear, our proof of the lower bound below is somewhat simpler and more intuitive if we assume that the first possible swap $T_l$ is $(1,2)$ -- equivalently, if we work with the modified problem where we consider unordered pairs $\{i,j\}$ instead of ordered ones. We recommend that the reader focuses on this case.

Let $T_1,\dots, T_l$ be a sequence of transpositions such that for any $(i,j)$, the pair $(1,2)$ is mapped to $(i,j)$ under some subsequence. Observe that all the relevant information about the first $t$ possible swaps $T_{l-t+1},T_{l-t+2},\dots,T_l$ is carried by a directed graph $G_t$ on vertex set $[n]$ with edges
$$E(G_t)=\{(i,j):\textnormal{there is a subsequence of $T_{l-t+1},\dots,T_l$ mapping $(1,2)$ to $(i,j)$}\}.$$
(In the simplified setting, this becomes an undirected graph.) We can describe the result of adding an additional transposition $T_{l-t}=(a,b)$ (to the beginning of this sequence) in terms of these graphs. Indeed, $E(G_{t+1})$ consists of the following edges.
\begin{itemize}
	\item All edges in $E(G_t)$;
	\item All pairs $(a,j)$ such that $(b,j)\in E(G_t)$, $j\not =a$;
	\item All pairs $(j,a)$ such that $(j,b)\in E(G_t)$, $j\not =a$;	
	\item All pairs $(b,j)$ such that $(a,j)\in E(G_t)$, $j\not =b$;
	\item All pairs $(j,b)$ such that $(j,a)\in E(G_t)$, $j\not =b$;
	\item The edges $(a,b)$ and $(b,a)$, provided that at least one of them appear in $G_t$.
\end{itemize}

Our proof relies on finding an appropriate invariant based on these digraphs. For any digraph $G$, let $f_1(G)$ denote the number of vertices in $G$ that appear in an edge (i.e., have positive in-degree or positive out-degree). Also, let us say that $X\subseteq V(G)$ is a \emph{nice clique} if for all $x,x'\in X$ (with $x'\not =x$), at least one of $(x,x')$ and $(x',x)$ appear in $E(G)$, and furthermore every $x\in X$ has both positive in-degree and positive out-degree in $G$ (but not necessarily in $X$). (In the simplified setting, this is just the usual notion of a clique in a simple graph.) Let us write $f_2(G)$ for the maximal size of a nice clique in $G$. The next lemma shows that $F(t)=f_1(G_t)+\frac{1}{2}f_2(G_t)$ increases by at most $1$ in each step. Note that this immediately gives the tight lower bound $\lceil 3n/2\rceil -2$, as $G_0$ is a single directed edge, and $G_l$ is a complete directed graph on $n$ vertices.

\begin{lemma}\label{lemma_12reachable}
	For all $0\leq t\leq l-1$, we have
	$$F(t+1)\leq F(t)+1.$$
\end{lemma}
\begin{proof}
	Let $T_{l-t}$ be the transposition $(a,b)$. For all $i$, let $S_i$ denote the set of vertices in $G_i$ appearing in an edge. Observe that
	\begin{equation*}
		S_{t+1}=\begin{cases*}
			S_t & if $a,b$ are both elements or both non-elements of $S_t$\\
			S_t\cup\{a\} & if $a\not \in S_t$, $b\in S_t$\\
			S_t\cup\{b\}& if $b\not \in S_t$, $a\in S_t$.
		\end{cases*}
	\end{equation*}
In particular, $f_1(G_{t+1})\leq f_1(G_t)$+1. Let us take a nice clique $X\subseteq [n]$ of maximal possible size $f_2(G_{t+1})$ in $G_{t+1}$.

Assume first that $f_1(G_{t+1})=f_1(G_t)+1$. Then, without loss of generality, $a\not \in S_t$ and $b\in S_t$. Hence $(a,b),(b,a)\not \in E(G_{t+1})$ and so $\{a,b\}\not\subseteq X$. If $a\not \in X$, then $X$ is a nice clique in $G_t$, thus $f_2(G_t)=f_2(G_{t+1})$, giving the result. However, if $a\in X$, then $(X\setminus \{a\})\cup\{b\}$ is a nice clique in $G_t$, once again giving $f_2(G_t)=f_2(G_{t+1})$ and hence $F_2(t+1)=F_2(t)+1$.

Now assume that $f_1(G_{t+1})=f_1(G_t)$. Observe that $X\setminus\{a,b\}$ is a nice clique in $G_{t}$, so we immediately get $f_2(G_t)\geq f_2(G_{t+1})-2$ and the result follows.
\end{proof}

\begin{proof}[Proof of Theorem~\ref{theorem_12reachable}]
	The lower bound $R_2(n)\geq \lceil 3n/2\rceil -2$ follows immediately from Lemma~\ref{lemma_12reachable}, since $F(0)=2$ and $F(l)=3n/2$.
	
	For the upper bound, it is easy to see that $R_2(n+1)\leq R_2(n)+2$ for all $n$, so it suffices to show that $R_2(n)\leq 3n/2-2$ when $n$ is even. Let $X=\{x_1,\dots,x_{n/2-1}\}$ and $Y=\{y_1,\dots,y_{n/2-1}\}$ partition $[n]\setminus\{1,2\}$. Let $T$ be the swap $(1,2)$, furthermore, let $T_i$ denote the transposition $(1,x_i)$, and similarly let $T_i'$ denote $(2,y_i)$. Finally, let $W_i$ denote $(x_i,y_i)$. It is easy to check that
	$$W_1,W_2,\dots W_{n/2-1},T_1',T_2',\dots T_{n/2-1}',T_1,T_2,\dots T_{n/2-1},T$$
	has a subsequence mapping $(1,2)$ to $(i,j)$, for any pair $(i,j)$. The result follows.
\end{proof}

Let us now turn to the proof of Theorem~\ref{theorem_12uniform}, showing $U_2(n)=2n-3$. For this problem, the random permutation obtained after $t$ swaps can be described by a weighted directed graph -- i.e., a matrix (with zeros on the diagonal). It turns out that a similar proof works if we replace the invariant $f_2(t)$ by the rank of this matrix.\begin{proof}[Proof of Theorem~\ref{theorem_12uniform}]
	The upper bound follows from \eqref{equation_knshufflebound}, so it is enough to prove the lower bound $U_2(n)\geq 2n-3$. Assume that $T_1,\dots,T_l$ is a $(2,n)$-shuffle. For each $1\leq t\leq l$, let us write $\sigma_t=T_{l-t+1}T_{l-t+2}\dots T_l$ for the random permutation obtained from the first $t$ lazy swaps, and let $\sigma_0$ be the identity. So we know that $\sigma_l$ is uniformly distributed in $S_n$. For each $t$, we form a matrix $M^{(t)}$ with entries
	
	$$M^{(t)}_{i,j}=\PP(\sigma(1)=i,\sigma(2)=j).$$
	Note that the matrix $M^{(t)}$ carries all the relevant information coming from the first $t$ lazy transpositions. Moreover, we have
\begin{equation*}
	M^{(0)}=
\left(
{\begin{array}{ccccc}
		0 & 1 & 0 & \dots & 0\\
		0 & 0 & 0 & \dots & 0\\
				0 & 0 & 0 & \dots & 0\\
		&&\dots&&\\
		0&0&0&\dots&0
\end{array} } \right),
\end{equation*}
and
\begin{equation*}
	M^{(l)}=
	\left(
	{\begin{array}{ccccc}
			0 & x & x & \dots & x\\
			x & 0 & x & \dots & x\\
			x & x & 0 & \dots & x\\
			&&\dots&&\\
			x&x&x&\dots&0
	\end{array} } \right),
\end{equation*}
where $x=\frac{1}{n(n-1)}$.

For each $t$, let $S(t)\subseteq [n]$ be the set of all $i\in[n]$ such that the $i$th row or the $i$th column of $M^{(t)}$ contains a non-zero element. In other words, $S(t)$ consists of elements in $[n]$ that can be reached by $1$ or $2$ after $t$ lazy swaps. Let us also write
$$f(t)=|S(t)|+\rank(M^{(t)}).$$
It is easy to see that $f(0)=3$ and $f(l)=2n$. So
the result follows immediately from the following claim.

\textbf{Claim.} For each $0\leq t\leq l-1$, we have $$f(t+1)\leq f(t)+1.$$

\textbf{Proof.} Fix some $t$ with $0\leq t\leq l-1$, and let us write $M=M^{(t)}$, $M'=M^{(t+1)}$. Let $T_{l-t}$ have parameters $(a,b,p)$ (where $a\not =b$). Observe that we have 

\begin{equation*}
	M'_{i,j}=
	\begin{cases*}
		M_{i,j} & if $i\not \in \{a,b\}$ and $j\not \in \{a,b\}$,\\
		(1-p)M_{i,j}+pM_{\bar{i},j}& if $i\in\{a,b\}$, $j\not \in\{a,b\}$ and $\{a,b\}=\{i,\bar{i}\}$,\\
		(1-p)M_{i,j}+pM_{i,\bar{j}}& if $i\not\in\{a,b\}$, $j \in\{a,b\}$ and $\{a,b\}=\{j,\bar{j}\}$,\\
		0&if $i=j\in\{a,b\}$,\\
		(1-p)M_{i,j}+pM_{j,i}& if $\{i,j\}=\{a,b\}$.
	\end{cases*}
\end{equation*}

Let $P$ be the following $n\times n$ matrix:
\begin{equation*}
	P_{i,j}=
	\begin{cases*}
	1 & if $i=j\not \in \{a,b\}$,\\
	0 & if $i\not =j$ and $|\{i,j\}\cap\{a,b\}|\leq 1$,\\
	1-p & if $i=j\in\{a,b\}$,\\
	p & if $\{i,j\}=\{a,b\}$.
	\end{cases*}
\end{equation*}
For example, if $a=1$ and $b=2$, then $P$ is given by
\begin{equation*}
	P=
	\left(
	{\begin{array}{cccccc}
			1-p & p & 0 & 0 &\dots & 0\\
			p & 1-p & 0 & 0 & \dots & 0\\
			0 & 0 & 1 &  0 &\dots & 0\\
			0&0&0&1&\dots&0\\
			0&0&0&0&\dots&1
	\end{array} } \right).
\end{equation*}
Using that $M_{a,a}=M_{b,b}=0$, we see that the matrix $PMP$ has entries given by

\begin{equation*}
	(PMP)_{i,j}=
	\begin{cases*}
		M_{i,j} & if $i\not \in \{a,b\}$ and $j\not \in \{a,b\}$,\\
		(1-p)M_{i,j}+pM_{\bar{i},j}& if $i\in\{a,b\}$, $j\not \in\{a,b\}$ and $\{a,b\}=\{i,\bar{i}\}$,\\
		(1-p)M_{i,j}+pM_{i,\bar{j}}& if $i\not\in\{a,b\}$, $j \in\{a,b\}$ and $\{a,b\}=\{j,\bar{j}\}$,\\
		p(1-p)(M_{a,b}+M_{b,a})&if $i=j\in\{a,b\}$,\\
		(1-p)^2M_{i,j}+p^2M_{j,i}& if $\{i,j\}=\{a,b\}$.
	\end{cases*}
\end{equation*}
So we have $M'=PMP+X$, where
\begin{equation*}
X_{i,j}=
\begin{cases*}
-p(1-p)(M_{a,b}+M_{b,a})&if $i=j\in\{a,b\}$,\\
p(1-p)(M_{a,b}+M_{b,a})& if $\{i,j\}=\{a,b\}$,\\
0& otherwise.
\end{cases*}
\end{equation*}

Observe that $\rank(X)\leq 1$, and if $\rank(X)=1$ then $M_{a,b}\not =0$ or $M_{b,a}\not =0$. Since $M'=PMP+X$, it follows that $\rank(M')\leq \rank(M)+1$, with equality only if $M_{a,b}\not =0$ or $M_{b,a}\not =0$.

On the other hand, it is easy to see that for each $i\not \in\{a,b\}$, we have $i\in S(t+1)$ if and only if $i\in S(t)$. Furthermore, $a\in S(t+1)$ only if $a\in S(t)$ or $b\in S(t)$, and similarly $b\in S(t+1)$ only if $a\in S(t)$ or $b\in S(t)$. It follows that $|S(t+1)|\leq |S(t)|+1$, with equality only if exactly one of $a$ and $b$ belong to $S(t)$. In particular, we need $M_{a,b}=M_{b,a}=0$ for equality.

It follows that $|S(t+1)|+\rank(M')\leq |S(t)|+\rank(M)+1$, proving the claim and hence the theorem.
\end{proof}
\section{Strong shuffles}
In this section we prove Theorems~\ref{theorem_strongshufflefor1} and~\ref{theorem_strongshufflefor2} about strong $(k,n)$-shuffles, i.e., when we want all $k$-tuples to have uniformly random image. We will begin with the case $k=1$.
\subsection{Strong \texorpdfstring{$(1,n)$}{(1,n)}-shuffles}
Let us start by giving a construction of a strong $(1,n)$-shuffle of length $\frac{1}{2}n\log_2{n}$ for the case when $n$ is a power of $2$. That is, we show that $\frac{1}{2}n\log_2{n}$ lazy swaps suffice to make the image of every element uniform. We arrange the $n=2^t$ elements of $[n]$ along the vertices of a $t$-dimensional hypercube $\{0,1\}^t$. Then our lazy swaps come in $t$ phases, each of length $2^{t-1}$, with the $i$th phase consisting of all possible swaps in direction $i$. That is, if $e_i$ denotes the $t$-dimensional unit vector with $i$th coordinate $1$, then in the $i$th phase we perform all the lazy swaps $(v,v+e_i,1/2)$ (with $v\in\{0,1\}^t$ satisfying $v_i=0$), in an arbitrary order. It is easy to check that after these $2^{t-1}t$ lazy transpositions, every point can end up everywhere else with probability exactly $1/2^t$, giving the upper bound
$$\UU_1(2^t)\leq 2^{t-1}t.$$

	
This is in fact the exact value of $\UU_1(2^t)$. To prove this, we establish the
corresponding lower bound.
	
\begin{lemma}\label{lemma_strongshufflelowerbound}
	For any positive integer $n$, we have $\UU_1(n)\geq \frac{1}{2}n\log_2 n$.
\end{lemma}

The proof uses a rather unusual invariant: we consider the `heaviest transversal' in the matrix $A_{i,j}=\PP(i\textnormal{ maps to }j)$.
\begin{proof}
	Assume that $T_1,\dots,T_l$ form a strong $(1,n)$-shuffle. For each $1\leq i\leq l$, let us write $\sigma_i=T_{l-i+1}T_{l-i+2}\dots T_l$ for the random permutation obtained from the first $i$ lazy swaps, and let $\sigma_0$ be the identity.

For each $0\leq i\leq l$ and each $\alpha\in S_n$, let
$$g(i,\alpha)=\prod_{x=1}^{n}\PP(\sigma_i(x)=\alpha(x)),$$
and let 
$$g(i)=\max_{\alpha\in S_n} g(i,\alpha).$$

Observe that $g(0)=1$ and $g(l)=\frac{1}{n^n}$. So the lower bound
$$\UU_1(n)\geq \frac{1}{2}n\log_2 n$$
follows immediately from the following claim.

\textbf{Claim.} For all $0\leq i\leq l-1$, we have $$g(i+1)\geq \frac{1}{4}g(i).$$

\textbf{Proof.} Let $\alpha\in S_n$ be such that $g(i)=g(i,\alpha)$, and let $T_{l-i}$ have parameters $(a,b,p)$. Let $\beta=(a,b)\alpha$, $\sigma=\sigma_i$ and $\alpha^{-1}(a)=c$, $\alpha^{-1}(b)=d$. Then
\begin{align*}
	g(i+1,\alpha)&=\prod_{x=1}^{n}\PP(\sigma_{i+1}(x)=\alpha(x))\\
	&=\left[\prod_{x\not =c,d}\PP(\sigma(x)=\alpha(x))\right][(1-p)\PP(\sigma(c)=a)+p\PP(\sigma(c)=b)][(1-p)\PP(\sigma(d)=b)+p\PP(\sigma(d)=a)]\\
	&\geq\left[\prod_{x\not =c,d}\PP(\sigma(x)=\alpha(x))\right][(1-p)\PP(\sigma(c)=a)][(1-p)\PP(\sigma(d)=b)]\\
	&=(1-p)^2g(i,\alpha),		
\end{align*}
and similarly
\begin{align*}
	g(i+1,\beta)&=\prod_{x=1}^{n}\PP(\sigma_{i+1}(x)=\beta(x))\\
	&=\left[\prod_{x\not =c,d}\PP(\sigma(x)=\alpha(x))\right][p\PP(\sigma(c)=a)+(1-p)\PP(\sigma(c)=b)][p\PP(\sigma(d)=b)+(1-p)\PP(\sigma(d)=a)]\\
	&\geq\left[\prod_{x\not =c,d}\PP(\sigma(x)=\alpha(x))\right][p\PP(\sigma(c)=a)][p\PP(\sigma(d)=b)]\\
	&=p^2g(i,\alpha).
\end{align*}
But either $p^2\geq 1/4$ or $(1-p)^2\geq 1/4$, giving $\max\{g(i+1,\alpha),g(i+1,\beta)\}\geq \frac{1}{4}g(i,\alpha)=\frac{1}{4}g(i)$, proving the claim and hence the lemma.
\end{proof}

Curiously, giving an upper bound in the case when $n$ is not a power of 2 is
considerably more difficult that the construction in the power of 2 case. We will
need the following lemma.

\begin{lemma}\label{lemma_strongshuffleupperbound}
For any positive integers $n$ and $r$, we have
\begin{equation*}
	\UU_1(n+r)\leq \UU_1(n)+\UU_1(r)+n+r-1.
\end{equation*}
\end{lemma}
\begin{proof}
We can take $\UU_1(n)$ lazy transpositions permuting $\{1,\dots,n\}$ only such that their product $\sigma$ satisfies $\PP(\sigma(i)=j)=1/n$ for all $i,j\in [n]$, and similarly, we can take $\UU_1(r)$ lazy transpositions permuting $n+1,\dots,n+r$ only such that their product $\rho$ satisfies $\PP(\rho(i)=j)=1/r$ for all $i,j\in [n+1,n+r]$. We claim that there exist lazy transpositions $T_1,\dots,T_l$ with $l\leq n+r-1$ such that $\tau=T_lT_{l-1}\dots T_1\rho \sigma$ satisfies $\PP(\tau(i)=j)=1/(n+r)$ for all $i,j\in [n+r]$.
	
	We recursively construct the lazy transpositions $T_1,\dots,T_l$ such that for all $0\leq t\leq l$, writing $f_t$ for the random permutation $T_t T_{t-1}\dots T_1$, there is some $j\in[n+1]$ and $j'\in \{n+1,\dots,n+r+1\}$ with the following properties.
	\begin{itemize}
		\item If $1\leq i<j$ or $n+1\leq i<j'$, then $\PP(f_t^{-1}(i)\in [n])=\frac{n}{n+r}$.
		\item If $j<i\leq n$ or $j'<i\leq n+r$, then $T_1,\dots,T_t$ all fix $i$.
		\item Either $j=n+1$ and $j'=n+r+1$; or $j\not =n+1$, $j'\not= n+r+1$, $\PP(f_t^{-1}(j)\in [n])>\frac{n}{n+r}$ and $\PP(f_t^{-1}(j')\in [n])<\frac{n}{n+r}$.
		\item We have $j+(j'-n)\geq t+2$.
	\end{itemize}
	When $t=0$, these are satisfied for $j=1$, $j'=n+1$. Now assume that $T_1,\dots, T_t$ are already constructed and satisfy the conditions above. If we have $j=n+1$ and $j'=n+r+1$, the process terminates (and we set $l=t$). Otherwise we construct $T_{t+1}$ as follows. Let us write $q=\PP(f_t^{-1}(j)\in [n])$ and $q'=\PP(f_t^{-1}(j')\in [n])$, so we know $q>n/(n+r)>q'$.
	
	If $q+q'>2n/(n+r)$, let $p=\frac{n/(n+r)-q'}{q-q'}.$ Note that $0<p<1$, $(1-p)q'+pq=n/(n+r)$, and $(1-p)q+pq'=q+q'-((1-p)q'+pq)>n/(n+r)$. So if we define $T_{t+1}$ to be the lazy transposition with parameters $(j,j',p)$, then the conditions above will be satisfied when $(t,j,j')$ is replaced by $(t+1,j,j'+1)$. (Note that we cannot have $j'+1=n+r+1$, otherwise $n=\sum_{i\in[n+r]}\PP(f_{t+1}^{-1}(i)\in [n])>n\cdot n/(n+r)+r\cdot n/(n+r)=n$, giving a contradiction.)
	
	Similarly, if $q+q'<2n/(n+r)$, then let $p=\frac{q-n/(n+r)}{q-q'}$, so we have $0<p<1$, $(1-p)q+pq'=n/(n+r)$ and $(1-p)q'+pq<n/(n+r)$. So if $T_{t+1}$ has parameters $(j,j',p)$, then the conditions are satisfied when $(t,j,j')$ is replaced by $(t+1,j+1,j')$.
	
	Finally, if $q+q'=2n/(n+r)$, then let $T_{t+1}$ have parameters $(j,j',1/2)$, so then the conditions are satisfied for $(t+1,j+1,j'+1)$. This finishes the recursive construction.
	
	So let us take $T_1,\dots, T_l$ as above. Note that for $t=l-1$ we have $j\leq n$ and $j'\leq n+r$. By the last property, we get $l\leq n+r-1$. Furthermore, by the first property (for $t=l$), we have $\PP(f_l^{-1}(i)\in [n])=\frac{n}{n+r}$ for all $i$. It follows that $\PP(f_l\rho\sigma(x)=y)=1/(n+r)$ for all $x,y\in [n+r]$, proving the claim.
\end{proof}

\begin{proof}[Proof of Theorem~\ref{theorem_strongshufflefor1}]
The lower bound follows from Lemma~\ref{lemma_strongshufflelowerbound}, and we have seen that we get the corresponding upper bound $\UU_1(n)\leq \frac{1}{2}n\log_2 n$ when $n$ is a power of $2$. So we have
 $$\UU_1(2^t)=2^{t-1}t.$$

We now prove the upper bound in the case when $n$ is not a power of $2$. Let $n$ have binomial expansion $$n=\sum_{i=0}^{\lfloor \log_2 n\rfloor}2^i\epsilon_i$$ (where $\epsilon_i\in\{0,1\}$ for all $i$). By
Lemma~\ref{lemma_strongshuffleupperbound}, we have
$$\UU_1(s+2^i)\leq \UU_1(s)+\UU_1(2^i)+2^i+s-1$$
for all $s$. Using this several times with $s=\sum_{j<i}2^j\epsilon_j$, we get
\begin{align*}
	\UU_1(n)&\leq \sum_{i:\epsilon_i=1}(\UU_1(2^i)+\sum_{j\leq i}{2^j\epsilon_j}-1)\\
	&\leq \sum_{i:\epsilon_i=1}(\UU_1(2^i)+2^{i+1})\\
	&=\sum_{i:\epsilon_i=1}(2^{i-1} i+2^{i+1})\\
	&\leq \sum_{i:\epsilon_i=1}(2^{i-1} \log_2 n+2^{i+1})\\
	&=\frac{1}{2}n\log_2 n+2n,
\end{align*}
as claimed.
\end{proof}

\subsection{Strong \texorpdfstring{$(2,n)$}{(2,n)}-shuffles}
We now turn to the proof of Theorem~\ref{theorem_strongshufflefor2}. Our first step is to introduce the notion of a `division shuffle'; this will be crucial for our construction. Given a positive integer $n$ with $n$ even, 
let us say that the lazy transpositions $T_1,\dots, T_l$ form a \emph{division $(2,n)$-shuffle} (of length $l$) if whenever $i,j\in [n]$ are distinct, then
\begin{align*}
	\PP(T_1\dots T_l(i)\in [n/2]\textnormal{ and } T_1\dots T_l(j)\in [n/2])&=\frac{(n/2)(n/2-1)}{n(n-1)}=\frac{1}{4}-\frac{1}{4(n-1)},\\
	\PP(T_1\dots T_l(i)\not \in [n/2]\textnormal{ and } T_1\dots T_l(j)\in [n/2])&=\frac{(n/2)^2}{n(n-1)}=\frac{1}{4}+\frac{1}{4(n-1)}, \\
	\PP(T_1\dots T_l(i)\not \in [n/2]\textnormal{ and } T_1\dots T_l(j)\not \in [n/2])&=\frac{(n/2)(n/2-1)}{n(n-1)}=\frac{1}{4}-\frac{1}{4(n-1)}.
\end{align*}
In other words, whenever we pick two elements $i,j$ in $[n]$, then the image under $T_1\dots T_l$ of the pair $(i,j)$ is distributed between $\{1,\dots,n/2\}$ and $\{n/2+1,\dots,n\}$ in the same way as a random pair coming from $[n]$. Let $h_0(n)$ denote the shortest possible length of a division $(2,n)$-shuffle. It is not difficult to see that $$\UU_2(2n)\leq h_0(n)+2\UU_2(n),$$ so it suffices to bound $h_0(n)$. We will do this by an inductive argument, but it will be convenient to use a slightly stronger property.

If a division $(2,n)$-shuffle $T_1,\dots,T_l$ is also a strong $(1,n)$-shuffle, i.e., for all $i,j\in [n]$ we have
\begin{equation*}
		\PP(T_1\dots T_l(i)=j)=1/n,
\end{equation*}
let us say that $T_1,\dots, T_l$ is a \emph{nice division $(2,n)$-shuffle}. Let $h(n)$ be the shortest possible length of a nice division $(2,n)$-shuffle. We will prove the following result.

\begin{lemma}\label{lemma_hinductive}
	If $n$ is even, then $$h(2n)\leq 2h(n)+2n.$$
\end{lemma}

Before we prove this lemma, we informally describe our 
construction. We divide the elements of $[2n]$ into two (equal sized) groups: `top' and `bottom' points. Furthermore, we further divide each of the top and bottom groups into `left' and `right'. We start with $h(n)$ lazy transpositions on the bottom points so that the image of any pair coming from the bottom will be divided between left and right in the same way as a random pair. We do the same thing for top vertices. Then, for some $q\in (0,1)$, we will swap top-bottom pairs on the left with probability $q$, and top-bottom pairs on the right with probability $1-q$, see Figure~\ref{figure_divisionshuffle}. It can be shown that there is some $q$ such that we end up with a division $(2,2n)$-shuffle. Indeed, if $q=1/2$, then any two points which start in the same part of the top-bottom division are too likely to end up in the same group, whereas they are too likely to end up in opposite groups if $q=0$, so by continuity we can pick an appropriate value of $q$. The construction finishes with $n$ additional lazy transpositions guaranteeing that our division $(2,2n)$-shuffle is also a strong $(1,2n)$-shuffle.

	\begin{figure}[h]
	\begin{tikzpicture}[node distance={30mm}, thick, main/.style={fill,circle}]
		
		\node[main, label=below:{$\frac{n}{2}$}](b1) at (-1.3,0){};
		\node[main, label=below: {$\dots$}](b2) at (-2.6,0){};
		\node[main, label=below: {$2$}](b3) at (-3.9,0){};
		\node[main, label=below:{$1$}](b4) at (-5.2,0){};
		\node[main, label=below: {$\frac{n}{2}+1$}](b5) at (1.3,0){};
		\node[main, label=below: {$\frac{n}{2}+2$}](b6) at (2.6,0){};
		\node[main,label=below: {$\dots$}](b7) at (3.9,0){};
		\node[main,label=below: {$n$}](b8) at (5.2,0){};
\node[main, label=above:{$\frac{3n}{2}$}](t1) at (-1.3,2){};
\node[main, label=above: {$\dots$}](t2) at (-2.6,2){};
\node[main, label=above: {$n+2$}](t3) at (-3.9,2){};
\node[main, label=above:{$n+1$}](t4) at (-5.2,2){};
\node[main, label=above: {$\frac{3n}{2}+1$}](t5) at (1.3,2){};
\node[main, label=above: {$\frac{3n}{2}+2$}](t6) at (2.6,2){};
\node[main,label=above: {$\dots$}](t7) at (3.9,2){};
\node[main,label=above: {$2n$}](t8) at (5.2,2){};
\draw (t1)--(b1)node[draw=none,fill=none,font=\scriptsize,midway,left] {$q$};
\draw (t2)--(b2)node[draw=none,fill=none,font=\scriptsize,midway,left] {$q$};
\draw (t3)--(b3)node[draw=none,fill=none,font=\scriptsize,midway,left] {$q$};
\draw (t4)--(b4)node[draw=none,fill=none,font=\scriptsize,midway,left] {$q$};
\draw (t5)--(b5)node[draw=none,fill=none,font=\scriptsize,midway,right] {$1-q$};
\draw (t6)--(b6)node[draw=none,fill=none,font=\scriptsize,midway,right] {$1-q$};
\draw (t7)--(b7)node[draw=none,fill=none,font=\scriptsize,midway,right] {$1-q$};
\draw (t8)--(b8)node[draw=none,fill=none,font=\scriptsize,midway,right] {$1-q$};
	\end{tikzpicture}
	\centering
	\caption{To obtain our division $(2,2n)$-shuffle, we divide the points into top/bottom and left/right. We perform nice division $(2,n)$-shuffles at the top and at the bottom, and then take the lazy transpositions shown on this figure, for some appropriately chosen value of $q$.}
	\label{figure_divisionshuffle}
\end{figure}
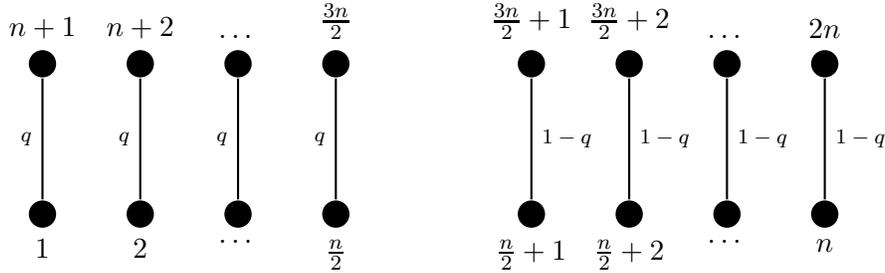
\begin{proof}
	Let $T_1,\dots, T_l$ be a nice division $(2,n)$-shuffle of length $l=h(n)$. If $T_i$ has parameters $(a,b,p)$ (with $a,b\in [n]$), let $T_i'$ be the lazy transposition with parameters $(n+a,n+b,p)$. So $T_1',\dots, T_l'$ form a nice division $(2,n)$-shuffle on ground set $[n+1,2n]$. Let us write $\sigma$ for the random permutation $T_1\dots T_l$, and similarly let $\sigma'=T_1'\dots T_l'$.
	
	Let $q$ be some number in $(0,1)$ (specified later). For each $1\leq i\leq n/2$, let $S_i$ be the lazy transposition with parameters $(i,n+i,q)$, and for each $n/2+1\leq i\leq n$, let $S_i$ have parameters $(i,n+i,1-q)$. Let $\rho=S_1\dots S_n$.
	
	Observe that if $i,j$ are distinct elements of $[2n]$ such that either $i,j\in [n]$ or $i,j\in [n+1,2n]$, then
	
	\begin{align*}
		\PP(\rho\sigma'\sigma(i)\in [n]\textnormal{ and } \rho\sigma'\sigma(j)\in [n])&=\left(\frac{1}{4}-\frac{1}{4(n-1)}\right)(q^2+(1-q)^2)+\left(\frac{1}{4}+\frac{1}{4(n-1)}\right)2q(1-q)\\
		&=\frac{1}{4}-\frac{1}{4(n-1)}(1-4q(1-q)),\\
		\PP(\rho\sigma'\sigma(i)\not \in [n]\textnormal{ and } \rho\sigma'\sigma(j)\in [n])&=\left(\frac{1}{4}-\frac{1}{4(n-1)}\right)2q(1-q)+\left(\frac{1}{4}+\frac{1}{4(n-1)}\right)(q^2+(1-q)^2)\\
		&=\frac{1}{4}+\frac{1}{4(n-1)}(1-4q(1-q)),\\
		\PP(\rho\sigma'\sigma(i)\not \in [n]\textnormal{ and } \rho\sigma'\sigma(j)\not \in [n])&=\frac{1}{4}-\frac{1}{4(n-1)}(1-4q(1-q)).
	\end{align*}
Similarly, if $i\in [n]$ and $j\in [n+1,2n]$ or vice versa, then
	\begin{align*}
	\PP(\rho\sigma'\sigma(i)\in [n]\textnormal{ and } \rho\sigma'\sigma(j)\in [n])&=\frac{1}{2}q\left(\frac{n/2-1}{n}(1-q)+\frac{1}{2}q\right)+\frac{1}{2}(1-q)\left(\frac{n/2-1}{n}q+\frac{1}{2}(1-q)\right)\\
	&=\frac{1}{4}-\frac{1}{n}q(1-q),\\
	\PP(\rho\sigma'\sigma(i)\not\in [n]\textnormal{ and } \rho\sigma'\sigma(j)\in [n])&=\frac{1}{2}q\left(\frac{1}{n}+\frac{n/2-1}{n}q+\frac{1}{2}(1-q)\right)+\frac{1}{2}(1-q)\left(\frac{1}{n}+\frac{n/2-1}{n}(1-q)+\frac{1}{2}q\right)\\
	&=\frac{1}{4}+\frac{1}{n}q(1-q),\\
	\PP(\rho\sigma'\sigma(i)\not\in [n]\textnormal{ and } \rho\sigma'\sigma(j)\not\in [n])	&=\frac{1}{4}-\frac{1}{n}q(1-q).\\
\end{align*}

Observe that $\frac{n}{4(2n-1)}\in (0,1/4)$, so we can pick $q\in (0,1/2)$ such that

$$q(1-q)=\frac{n}{4(2n-1)}.$$
It follows from the equations above that for this particular choice of $q$, we have, for all $i,j\in [2n]$ (distinct),
\begin{align*}
	\PP(\rho\sigma'\sigma(i)\in [n]\textnormal{ and } \rho\sigma'\sigma(j)\in [n])&=\frac{1}{4}-\frac{1}{4(2n-1)},\\
	\PP(\rho\sigma'\sigma(i)\not \in [n]\textnormal{ and } \rho\sigma'\sigma(j)\in [n])&=\frac{1}{4}+\frac{1}{4(2n-1)}, \\
	\PP(\rho\sigma'\sigma(i)\not \in [n]\textnormal{ and } \rho\sigma'\sigma(j)\not \in [n])&=\frac{1}{4}-\frac{1}{4(2n-1)}.
\end{align*}
So $S_1,\dots S_n, T_1',\dots, T_l', T_1,\dots,T_l$ form a division $(2,2n)$-shuffle. 

Let $W_i$ be the lazy transposition $(i,n/2+i,1/2)$ for $i\in [n/2]\cup [n+1,n+n/2]$. It is easy to check that $$W_1,\dots, W_{n/2}, W_{n+1},W_{n+2},\dots,W_{n+n/2},S_1,\dots S_n, T_1',\dots, T_l', T_1,\dots,T_l$$ is both a strong $(1,2n)$-shuffle and a division $(2,2n)$-shuffle. The result follows.
\end{proof}

\begin{proof}[Proof of Theorem~\ref{theorem_strongshufflefor2}]
Observe that whenever $n\geq 2$ is even, then
\begin{equation}
	h(n+2)\leq h(n)+2n+1.\label{equation_hincreaseby1}
\end{equation}
Indeed, assume that $T_1,\dots,T_l$ is a division $(2,n)$-shuffle (fixing $n+1$ and $n+2$), and let $S_1,\dots, S_{2n+1}$ be lazy transpositions such that $S_1\dots S_{2n+1}$ maps $(n+1,n+2)$ uniformly to the pairs from $[n+2]$. Then $\sigma=S_{2n+1}\dots S_1$ satisfies $\PP(\sigma^{-1}(n+1)=i,\sigma^{-1}(n+2)=j)=\frac{1}{(n+2)(n+1)}$ for all $i,j\in [n+2]$ distinct. Conditioning on whether or not $\sigma(i)$ and $\sigma(j)$ belong to $\{n+1,n+2\}$, we see that for $\rho=T_1\dots T_l\sigma$ we have
\begin{align*}
	\PP(\rho(i)\in [n/2]\cup\{n+1\}\textnormal{ and } \rho(j)\in [n/2]\cup\{n+1\})&=2\frac{n}{(n+2)(n+1)}\frac{1}{2}+\frac{n(n-1)}{(n+2)(n+1)}\left(\frac{1}{4}-\frac{1}{4(n-1)}\right)\\
	&=\frac{1}{4}-\frac{1}{4(n+1)},\\
	\PP(\rho(i)\not\in [n/2]\cup\{n+1\}\textnormal{ and } \rho(j)\in [n/2]\cup\{n+1\})&=\frac{1}{4}+\frac{1}{4(n+1)},\\
	\PP(\rho(i)\not\in [n/2]\cup\{n+1\}\textnormal{ and } \rho(j)\not\in [n/2]\cup\{n+1\})&=\frac{1}{4}-\frac{1}{4(n+1)}.
\end{align*}
Furthermore, $T_1,\dots,T_l,S_{2n+1},S_{2n},\dots,S_1$ is easily seen to be a strong $(1,n+2)$-shuffle. The bound~\eqref{equation_hincreaseby1}  follows easily.

We also know from Lemma~\ref{lemma_hinductive} that
\begin{equation}
	h(2n)\leq 2h(n)+2n \label{equation_hdouble}
\end{equation}

and hence

\begin{equation}
	h(2n+2)\leq 2h(n)+6n+1.\label{equation_hdoubleplus2}
\end{equation}

Clearly $h(2)=1$, so if follows from~\eqref{equation_hdouble} and~\eqref{equation_hdoubleplus2} that whenever $n$ is even, we have
\begin{equation}
h(n)\leq 3n\log_2 n.\label{equation_hbound}
\end{equation}

Observe that if $n$ is even, then we have
\begin{equation}\label{equation_htoU2}
\UU_2(n)\leq h(n)+2\UU_2(n/2).
\end{equation}
Indeed, assume that $T_1,\dots, T_l$ form a division $(2,n)$-shuffle, let $T_1',\dots,T_u'$ be a strong $(2,n/2)$-shuffle (fixing $n/2+1,\dots,2n$), and let $T_1'',\dots, T_v''$ be another strong $(2,n/2)$-shuffle on ground set $n/2+1,\dots, n$ (in particular, it fixes $1,\dots, n/2$). Then $$T_1'',\dots, T_v'',T_1',\dots, T_u',T_1,\dots, T_l$$ is easily seen to be a strong $(2,n)$-shuffle, giving~\eqref{equation_htoU2}.

Furthermore, for all $n$ we have
\begin{equation}
	\UU_2(n+1)\leq \UU_2(n)+n.\label{equation_Uincreaseby1}
\end{equation}
Indeed, if $T_1,\dots,T_l$ form a strong $(2,n)$-shuffle (fixing $n+1$) and $S_1,\dots, S_n$ satisfy $\PP(S_1\dots S_n(i)=n+1)=1/(n+1)$ for all $i$, then $T_1,\dots, T_l,S_1,\dots, S_n$ is easily seen to be a strong $(2,n+1)$-shuffle.

It follows from~\eqref{equation_hbound}, \eqref{equation_htoU2} and~\eqref{equation_Uincreaseby1} that

$$\UU_2(n)\leq 2\UU_2(n/2)+3n\log_2 n$$
if $n$ is even, and
$$\UU_2(n)\leq 2\UU_2((n-1)/2)+3n\log_2 n+n$$
if $n$ is odd. Since $\UU_2(1)=0$ (and $\UU_2(2)=1$), we get by induction that
$$\UU_2(n)\leq 4n(\log_2 n)^2$$
for all $n$, giving the result.
\end{proof}

\section{Open problems}
We finish this paper with some open problems. Despite the important recent progress by Groenland, Johnston, Radcliffe and Scott~\cite{groenland2022perfect}, there is still a large gap between the upper and lower bounds for the problem raised by Fitzsimons~\cite{fitzsimons} and Angel and Holroyd~\cite{angel2018perfect}, and it would be very interesting to close this gap.
\begin{question}[\cite{fitzsimons,angel2018perfect}]
What is the asymptotic behaviour of $U(n)?$
\end{question}

One of the main topics considered in this paper was determining the shortest possible length $\UU_k(n)$ of strong $(k,n)$-shuffles. We gave essentially tight bounds in the case $k=1$, and in the next case $k=2$ we gave an upper bound of $O(n\log^2 n)$. It would be interesting to decide whether or not $\UU_k(n)$ is `small' for all fixed values of $k$.

\begin{question}\label{question_strongshuffle}
Given a fixed positive integer $k$, do we have $\UU_k(n)=O(n^{1+\epsilon})$ for all $\epsilon>0$?
\end{question}

We believe that the answer to Question~\ref{question_strongshuffle} should be 
positive. However, a negative answer would also be very interesting, as it would necessarily give a significantly improved lower bound for $U(n)$.

Another related problem is to fully close the gap between the bounds for $\UU_2(n)$. Theorem~\ref{theorem_strongshufflefor1} gives a lower bound $\UU_2(n)\geq \UU_1(n)=\Theta(n\log n)$, whereas by Theorem~\ref{theorem_strongshufflefor2} we have $\UU_2(n)=O(n\log^2 n)$.

\begin{question}
Do we have $\UU_2(n)=\Theta(n\log n)$?
\end{question}

Finally, as mentioned in the introduction, it would be interesting to consider the reachability problem that is analogous to strong shuffles. Given some $k$ and $n$, let $R_k^\textnormal{all}(n)$ denote the minimal number of lazy transpositions $T_1,\dots,T_l$ such that $T_1\dots T_l$ maps each $k$-tuple from $[n]$ to any other $k$-tuple with positive probability. As mentioned before, the case $k=1$ is the well-known `gossiping dons' problem, and for general $k$ we get a generalisation of that question. As noted in the introduction (in the special case $k=2$), we have $R_k^\textnormal{all}(n)\leq 2U_k(n)$ for all $k,n$, and therefore $R_k^\textnormal{all}(n)=\Theta(n)$ for all $k$.

\begin{question}
What is the value of $R_k^\textnormal{all}(n)$, exactly or asymptotically?
\end{question}

Even in the case $k=2$ this seems to be a difficult problem. Perhaps the answer is
about $3n$, as noted in the introduction?

\bibliography{Bibliography}

\begin{thebibliography}{1}

\bibitem{angel2018perfect}
O.~Angel and A.~E. Holroyd.
\newblock Perfect shuffling by lazy swaps.
\newblock {\em Electronic Communications in Probability}, 23:1--11, 2018.

\bibitem{bollobas2004extremal}
B.~Bollob{\'a}s.
\newblock {\em Extremal graph theory}.
\newblock Academic Press, 1978.

\bibitem{bumby1981problem}
R.~T. Bumby.
\newblock A problem with telephones.
\newblock {\em SIAM Journal on Algebraic Discrete Methods}, 2(1):13--18, 1981.

\bibitem{czumaj2015random}
A.~Czumaj.
\newblock Random permutations using switching networks.
\newblock In {\em Proceedings of the forty-seventh annual ACM symposium on
  Theory of Computing}, pages 703--712, 2015.

\bibitem{fitzsimons}
J.~Fitzsimons.
\newblock What is the most efficient way to generate a random permutation from
  probabilistic pairwise swaps?
\newblock Theoretical Computer Science Stack Exchange.
\newblock \url{https://cstheory.stackexchange.com/q/5321}.

\bibitem{groenland2022perfect}
C.~Groenland, T.~Johnston, J.~Radcliffe, and A.~Scott.
\newblock Perfect shuffling with fewer lazy transpositions.
\newblock {\em arXiv preprint arXiv:2208.06629}, 2022.

\bibitem{groenland2022short}
C.~Groenland, T.~Johnston, J.~Radcliffe, and A.~Scott.
\newblock Short reachability networks.
\newblock {\em arXiv preprint arXiv:2208.06630}, 2022.

\end{thebibliography}
\bibliographystyle{abbrv}	

\end{document}